\documentclass[11pt,reqno]{amsart}
\usepackage{amsfonts}
\usepackage{amsmath,amsthm,amssymb,amsfonts,amscd}
\usepackage{mathrsfs}
\usepackage{lipsum}
\usepackage{bbding}
\usepackage{graphicx,latexsym}
\usepackage{hyperref}
\usepackage{color}

\newcommand{\bea}{\begin{eqnarray}}
\newcommand{\eea}{\end{eqnarray}}
\newcommand{\bna}{\begin{eqnarray*}}
\newcommand{\ena}{\end{eqnarray*}}

\numberwithin{equation}{section}

\renewcommand{\thefootnote}{\fnsymbol{footnote}}
\theoremstyle{plain}
\newtheorem{theorem}{Theorem}[section]
\newtheorem{lemma}{Lemma}[section]

\newtheorem{proposition}{Proposition}[section]

\theoremstyle{definition}

\newtheorem{notation}{Notation}
\newtheorem{remark}{Remark}

\newcommand\blfootnote[1]{%
  \begingroup
  \renewcommand\thefootnote{}\footnote{#1}%
  \addtocounter{footnote}{-1}%
  \endgroup
}

\renewcommand{\Im}{\operatorname{Im}}

\begin{document}
\title{Sums of Fourier coefficients involving theta series and Dirichlet characters}

	\author{Yanxue Yu}
	\date{}
	\address{School of Mathematics and Statistics, Shandong University, Weihai\\Weihai, Shandong
		264209, China}
    \email{yanxueyu@mail.sdu.edu.cn}

\begin{abstract}
Let $f$ be a holomorphic or Maass cusp forms for $ \rm SL_2(\mathbb{Z})$ with normalized Fourier coefficients $\lambda_f(n)$  and
    \bna
    r_{\ell}(n)=\#\left\{(n_1,\cdots,n_{\ell})\in \mathbb{Z}^2:n_1^2+\cdots+n_{\ell}^2=n\right\}.
    \ena
Let $\chi$ be a primitive Dirichlet character of modulus $p$, a prime. In this paper, we are concerned with obtaining nontrivial estimates for the sum
	\bna
    \sum_{n\geq1}\lambda_f(n)r_{\ell}(n)\chi(n)w\left(\frac{n}{X}\right)
    \ena
for any $\ell \geq 3$, where $w(x)$ be a smooth function compactly supported in $[1/2,1]$.
\end{abstract}

\keywords{Fourier coefficients, Dirichlet character, Theta series, Circle method}

\blfootnote{{\it 2010 Mathematics Subject Classification}: 11F11, 11F27, 11F30, 11F66}

\maketitle
\section{Introduction}
Let $\{\lambda_F(n):n\geq1\}$ be the Hecke eigenvalues of a $\rm GL_m$ automorphic form $F$, $\chi$ be a primitive Dirichlet character of modulus a prime $p$ and $w$ be a smooth function.  It is a classical and important problem in analytic number theory to study the cancellation in sums of the form
    \bea\label{eq:S}
    S(X)=\sum_{n\geq1}\lambda_F(n)\chi(n)w\left(\frac{n}{X}\right),
    \eea
which arises in a variety of contexts. Nontrivial bounds of this sum has been studied intensively by many authors with various applications in areas such as subconvexity.

It is widely believed that the Hecke eigenvalues $\{\lambda_F(n) : n \geq1\}$ and the Dirichlet character  $\{\chi(n) : n \geq 1\}$  are not correlated, in the following sense
    \bna
    \sum_{n\geq1}\lambda_F(n)\chi(n)w\left(\frac{n}{X_m}\right)\ll_{F,w,A} X_m^{1+\varepsilon}
    \ena
for $X_m\ll p^{\frac{m}{2}+\varepsilon}$.  (For the case where $X_m\gg p^{\frac{m}{2}+\varepsilon}$ one can apply the  functional equation of $L(s,F)$ to either produce a nontrivial estimate or to reduce that  case to the current setting.)

In applications, one is more concerned with a power-saving estimate
    \bea\label{eq:nontirvial}
    \sum_{n\geq1}\lambda_F(n)\chi(n)w\left(\frac{n}{X_m}\right)\ll_{F,w} X_m^{1-\delta}
    \eea
for some $\delta>0$.
In general, this can be a challenging problem. In fact, it would imply a subconvexity bound for the $L$-function $L(s,F)$ in the $p$-aspect. Nontrivial results towards \eqref{eq:nontirvial} have been known for several cases.

$\bullet$ If $F=f$ is a $\rm SL_2(\mathbb{Z})$ Hecke cusp form, Munshi \cite{M} first established this result in his influential note on the Burgess Bound, demonstrating the following bound
    \bna
    \sum_{n\geq1}\lambda_f(n)\chi(n)w\left(\frac{n}{X_2}\right)\ll p^{\frac{3}{8}+\varepsilon}X_2^{\frac{1}{2}}+p^{-\frac{9}{8}+\varepsilon}X_2^{\frac{3}{2}}+p^{\frac{3}{4}+\varepsilon}
    \ena
for $ p^{\frac{1}{4}}<X_2<p^{1+\varepsilon}$.
Further related works on the Burgess bound can be found in \cite{AHLS}, while studies on Weyl-type bounds are available in \cite{Gho}.

$\bullet$ If $F=\pi$ is a $\rm SL_3(\mathbb{Z})$ Maass cusp form, it was proved by Munshi \cite{M1} that one has
    \bna
    \sum_{n\geq1}\lambda_\pi(m,n)\chi(n)w\left(\frac{m^2n}{X_3}\right)\ll_{\pi,w} p^{\frac{3}{4}-\delta_1}X_3^{\frac{1}{2}}
    \ena
for  $p^{\frac{3}{2}-\vartheta}\ll X_3\ll p^{\frac{3}{2}+\vartheta}$, where $\vartheta>0$ is sufficiently small.

$\bullet$ For the case where $F=f\otimes g$, a Rankin-Selberg convolution of two $\rm SL_2(\mathbb{Z})$ cusp form $f$ and $g$,  Ghosh \cite{Gho2} recently  established the following bound
    \bna
    \sum_{n\geq1}\lambda_f(n)\lambda_g(n)\chi(n)w\left(\frac{n}{X_4}\right)\ll_{f,g,w} p^{\frac{22}{23}+\varepsilon}X_4^{\frac{1}{2}},
    \ena
as long as $ p^{\frac{44}{23}}\ll X_4\ll p^{2+\varepsilon}$.

Motivated by this and with applications in mind, we shall consider the analogues of \eqref{eq:S}. In order to state our results, let us firstly set up some  basic notations.

	Let $H_{\kappa}$ denote the set of normalized primitive holomorphic cusp forms of even integral weight $\kappa$ for ${\rm SL}_2(\mathbb{Z})$. For $f(z)\in H_{\kappa}$, it has a Fourier expansion at the cusp $\infty$ given by
	\bna
	f(z)=\sum_{n=1}^{\infty}\lambda_{f}(n)n^{\frac{k-1}{2}}e(nz), \quad \Im z>0,	
    \ena
where $e(x):=\exp(2\pi{i}x)$ is an additive character as usual. The famous Ramanujan-Petersson conjecture, which was proved  by Deligne \cite{De}, asserts that
    \bea\label{eq:KS}
    \lambda_{f}(n)\ll d(n)\ll n^\varepsilon
    \eea
for all $n \geq 1$. Here $d(n)$ is the Dirichlet divisor function.
 	Similarly, let $S_{\mu}$ be the set of normalized primitive Maass cusp forms of Laplace eigenvalue $1/4+\mu^{2},\,\mu>0$. Then $f(z)\in S_{\mu}$ has a Fourier expansion at the cusp infinity
	\bna	
    f(z)=\sqrt{y}\sum_{n\neq0}\lambda_{f}(n)K_{ir}(2\pi|n|y)e(nx).
	\ena
	However, Ramanujan conjecture is still open for $f\in S_{\mu}$. The current best bound is due to Kim and Sarnak \cite{KS-2003}
	\bna\label{eq-KS}
	|\lambda_{f}(n)|\ll n^{\frac{7}{64}+\varepsilon}.
	\ena
For any $f\in H_{\kappa}$ or $S_{\mu}$, the Fourier coefficients $\lambda_{f}(n)$ are real, especially $\lambda_{f}(1)=1$. They also satisfy the multiplicative property
\bna
\lambda_f(mn)=\sum_{d|(m,n)}\mu(d)\lambda_f\left(\frac{m}{d}\right)\lambda_f\left(\frac{n}{d}\right).
\ena

Let
\bna r_{\ell}(n)=\#\left\{(n_1,\cdots,n_{\ell})\in \mathbb{Z}^2:n_1^2+\cdots+n_{\ell}^2=n\right\}. \ena
Then $r_\ell(n)$ is the $n$-th  Fourier coefficient of the modular form $\theta^\ell(z)$, where $\theta(z)$ is the classical Jacobi theta series
\bna
\theta(z)=\sum_{n\in\mathbb{Z}}e(n^2z).
\ena
By means of the results of \cite[Theorem A, B]{Bat} and \cite[Corollary 11.3]{Iwan}, the individual arithmetic function $r_\ell(n)$ satisfies
\bea\label{r}
r_{\ell}(n)\ll_{\ell,\varepsilon}n^{\frac{\ell}{2}-1+\varepsilon}
\eea
for any $\varepsilon > 0$.

In this paper, we are interested in bounding the sum
    \bna\label{S}
    \mathcal{S}(X)=\sum_{n\geq1}\lambda_f(n)r_{\ell}(n)\chi(n)w\left(\frac{n}{X}\right),
    \ena
where $ \lambda_f(n)$ are the Fourier coefficients of $\rm SL_2(\mathbb{Z})$ Hecke holomorphic or Hecke-Maass cusp form $f$, $\chi$ is a primitive Dirichlet character of modulus a prime $p$ and $w(x)$ is a smooth function compactly supported in $[1/2,1]$. Applying Cauchy-Schwarz inequality, individual bound \eqref{r} and the Rankin-Selberg estimates
    \bea\label{CSRS}
    \sum_{n\ll x}|\lambda_f(n)|^2\ll_{f,\varepsilon}x^{1+\varepsilon},
    \eea
we have the trivial bound
    \bna
    \mathcal{S}(X)\ll X^{\frac{\ell}{2}+\varepsilon}.
    \ena
Our main theorem gives a nontrivial power saving over this estimate.
    \begin{theorem}\label{Th11}
    Let $f\in H_\kappa$ or $S_\mu$ and $\chi$ be a primitive Dirichlet character of modulus $p$, a prime. Let $w:\mathbb{R}\to[0,\infty)$ be a smooth function compactly supported in $[1/2, 1]$  and satisfying
        \bna
        w^{(j)}(x)\ll_j \Delta^j \quad\text{ for}\quad j\geq 0 \quad
        \text{and} \quad
        \int|w^{(j)}(\xi)|\mathrm{d}\xi\ll \Delta^{j-1} \quad \text{for}\quad j\geq1,
        \ena
    where $1\leq\Delta<X$. Then for any $\ell \geq 3$, we have
	   \bna
        \sum_{n\geq1}\lambda_f(n)r_{\ell}(n)\chi(n)w\left(\frac{n}{X}\right)
        \ll_{f,\ell,\varepsilon} p^{\frac{\ell-2}{\ell+3}}X^{\frac{\ell}{2}-\frac{\ell-2}{\ell+3}+\varepsilon}
        \Delta^{\frac{\ell-2}{2(\ell+3)}}
	   \ena
for $p< X$.
    \end{theorem}

Further, using the estimates of $\lambda_f(n)$ in short intervals, we can derive the following theorem.
\begin{theorem}\label{Th12}
Same notation and assumptions as in Theorem \ref{Th11}.
Then for any $\ell \geq 3$, we have
	\bna \sum_{n\leq X}\lambda_f(n)r_{\ell}(n)\chi(n)
    \ll_{f,\ell,\varepsilon}  p^{\frac{\ell-2}{2\ell+1}}X^{\frac{\ell}{2}-\frac{\ell-2}{2\ell+1}}	\ena
for $p< X$.
\end{theorem}

\begin{remark}
For the sake of exposition we will only present the case of holomorphic cusp forms for $\rm SL_2(\mathbb{Z})$ as the proof for Maass forms follows in a similar manner.
\end{remark}

\begin{notation}
Throughout the paper, the letters $q$, $\ell$, $m$ and $n$, with or without subscript, denote integers. The letters $\varepsilon$ and $A$  denote arbitrarily small and large positive constants, respectively, not necessarily the same at different occurrences. The symbol $\ll_{a,b,c}$ denotes that the implied constant depends at most on $a$, $b$, and $c$, and $q\sim C$ means $C<q\leq 2C$.
\end{notation}

\section{Preliminary}
In this section, we will briefly give some fundamental tools which play a crucial role in the proof of our main results. First we recall the Voronoi summation formula for $\rm GL_2$ (see \cite[(1.12), (1.15)]{MS}).

\begin{lemma}\label{Voro}
Let $f$ be a holomorphic form with normalized Fourier coeffcient $\lambda_f(n)$ of weight $\kappa\geq2$ for $\rm SL_2(\mathbb{Z})$. Suppose that $\phi(x)\in C_c^{\infty}(0, \infty)$. Let $a,q\in\mathbb{Z}$ with $q\neq0$, $ (a,q)=1$ and $a\overline{a}\equiv1\;(\text{{\rm mod }} q)$. Then
    \bea\label{Voronoi summmation}
    \sum_{n=1}^{\infty}\lambda_f(n)e\left(\frac{an}{q}\right)\phi(n)
    =q\sum_{n=1}^{\infty}\frac{\lambda_f(n)}{n}e\left(-\frac{\overline{a} n}{q}\right)\Phi\left(\frac{n}{q^2}\right),
    \eea
where for $\sigma>-1-(\kappa+1)/2$,
    \bea\label{integral}
    \Phi(x)= i^{\kappa-1}\frac{1}{2\pi^2}\int_{(\sigma)}(\pi^2 x)^{-s}\rho_f (s)\widetilde{\phi}(-s)\mathrm{d}s 
    \eea
with
    \bna\label{gamma}
    \rho_f (s)=\frac{\Gamma\left(\frac{1+s+(\kappa+1)/2}{2}\right)\Gamma\left(\frac{1+s+(\kappa-1)/2}{2}\right)}
    {\Gamma\left(\frac{-s+(\kappa+1)/2}{2}\right)
    \Gamma\left(\frac{-s+(\kappa-1)/2}{2}\right)}.
    \ena
Here $\widetilde{\phi}(s)=\int_{0}^{\infty}\phi(u)u^{s-1}\mathrm{d}u$ is the Mellin transform of $\phi$.
\end{lemma}

We also require an asymptotic expansion for $\Phi(x)$. The asymptotic expansion for that integral was established by Lin and Sun \cite[Lemma 3.2]{LS}.
    \begin{lemma}\label{in}
    For any fixed integer $J\geq1$ and $xX\gg 1$, we have
    \bna
    \Phi(x)=x \int_0^\infty (xy)^{-\frac{1}{4}}\phi(y)\sum_{j=0}^{J}\frac{c_{j} e(2 \sqrt{xy})+d_{j} e(-2 \sqrt{xy})}{(xy)^{\frac{j}{2}}}\mathrm{d}y+O_{\kappa,J}\left(x^{\frac{1}{4}-\frac{J}{2}}\right),
    \ena
where $c_{j}$ and $d_{j}$ are constants depending on $\kappa$.
\end{lemma}

Next we evaluate the integral  $\Phi$ in the Voronoi summation \eqref{Voronoi summmation}.
\begin{lemma}\label{evaluation of G}
Let $ \Phi(x)$ be defined as in \eqref{integral}. Let $\phi(x)$ be a fixed smooth function compactly supported on $[aX, bX]$ with $b > a > 0$ and satisfying
    \bna
    \phi^{(j)}(x)\ll_j (X/R)^{-j} \quad\text{ for}\; j\geq 0 \quad
    \text{and} \quad
    \int|\phi^{(j)}(\xi)|\mathrm{d}\xi \ll\frac{1}{(X/R)^{j-1}} \quad \text{for}\; j\geq1,
    \ena
for some $R\geq1$. Then we have
    \bna
    \Phi(x)=
    \left\{
    \begin{array}{ll}
    X^{-A} , & \text{if }\, x > R^{2+\varepsilon}X^{-1} ,\\
    (xX)^{\frac{1}{4}}, & \text{if }\, X^{-1}\ll x \leq R^{2+\varepsilon}X^{-1} ,\\
    (xX)^{\frac{1}{2}}R^\varepsilon, & \text{if }\, x \ll X^{-1}.
    \end{array} \right.
    \ena
\begin{proof}
For $s=\sigma+i\tau,\, \sigma\geq j-1,\, j\geq1$, by integration by parts $j$ times, one has
    \bna
    \widetilde{\phi}(-s)
    &=&\frac{1}{s(s-1)\cdots(s-j+1)}\int_{0}^{\infty}\phi^{(j)}(u)u^{-s+j-1}\mathrm{d}u\\
    &\ll&\frac{X^{-\sigma+j-1}}{(1+|\tau|)^j}\int_{0}^{\infty}|\phi^{(j)}(u)|\mathrm{d}u\\
    &\ll& \frac{X^{-\sigma}}{R}\left(\frac{R}{1+|\tau|}\right)^j.
    \ena
Thus by Stirling's approximation, we have
    \bna
    \Phi(x)
    &\ll&x^{-\sigma}\int_{-\infty}^{+\infty}(1+|\tau|)^{2(\sigma+\frac{1}{2})}\frac{X^{-\sigma}}{R}\left(\frac{R}{1+|\tau|}\right)^j\mathrm{d}\tau\\
    &\ll&(xX)^{-\sigma}R^{j-1}\int_{-\infty}^{+\infty}\frac{1}{(1+|\tau|)^{j-2\sigma-1}}\mathrm{d}\tau.
    \ena
Take an appropriate $\sigma$ such that $j=2\sigma+2+\varepsilon$ is an integer for some $\varepsilon>0$. Then we deduce that
    \bna
    \Phi(x)
    \ll(xX)^{-\sigma}R^{2\sigma+1+\varepsilon}
    \ll R^{1+\varepsilon}\left(\frac{R^2}{xX}\right)^\sigma.
    \ena
Thus we can write, for $x>R^{2+\varepsilon}X^{-1}$,
    \bna
    \Phi(x)\ll X^{-A}.
    \ena

For $xX\gg 1$, by virtue of Lemma \ref{in}, we infer that
    \bna
    \Phi(x)\ll x\sum_{j=0}^{J}\left|\int_{0}^{\infty}(xy)^{-\frac{j}{2}-\frac{1}{4}}\phi(y)
    \left(c_j e(2\sqrt{xy})+d_j e(-2\sqrt{xy})\right)\mathrm{d}y\right|+(xX)^{\frac{1}{4}-\frac{J}{2}}.
    \ena
By partial integration, it is easy to see that
    \bna
    \Phi(x)\ll \sum_{j=0}^{J}(xX)^{\frac{1}{4}-\frac{j}{2}}+(xX)^{\frac{1}{4}-\frac{J}{2}}
    \ll (xX)^{\frac{1}{4}}.
    \ena

For $xX\ll 1$, we take $ \sigma=-\frac{1}{2}$ to get
    \bna
    \Phi(x)\ll (xX)^{\frac{1}{2}}R^\varepsilon.
    \ena
This completes the proof of the lemma \ref{evaluation of G}.\\
\end{proof}
\end{lemma}

\begin{lemma}[Hua's lemma]\label{Hua's lemma}
	Suppose that $1\le j \le k$. Then
	\bna
	\int_{0}^{1}\Big|\sum_{1\le m\le N}e(\alpha m^k) \Big|^{2^j}d\alpha
    \ll N^{2^j-j+\varepsilon},
	\ena
where the implied constant depending on $ k$ and $ \varepsilon$.
    \begin{proof}
    See \cite[Lemma 2.5]{V} or  \cite[Lemma 20.6]{IK}.
    \end{proof}
\end{lemma}

\section{Proof of Theorem \ref{Th11}}
We will prove Theorem \ref{Th11}  by the classical circle method. For any $\alpha\in\mathbb{R}$, we define
    \bea\label{FG}
    F(\alpha)=\sum_{|m|\leq X^{1/2}}e(\alpha m^2),\qquad
    G(\alpha)=\sum_{n\geq 1}\lambda_f(n)\chi(n)e(-\alpha n)w\left(\frac{n}{X}\right).
    \eea
Then by the orthogonality relation
    \begin{center}	
	$\int_{0}^{1} e(n\alpha)d\alpha=\begin{cases}
	1,&\text{if } n=0,\\
	0,&\text{if } n\in \mathbb{Z}\backslash \{0\},
	
    \end{cases}$
    \end{center}
one has
    \bna
    \mathcal{S}(X)
    &=&\sum_{n\geq1}\lambda_f(n)\chi(n)\phi\left(\frac{n}{X}\right)
    \mathop{\sum_{|m_1|\leq X^{1/2}}\cdots\sum_{|m_{\ell}|\leq X^{1/2}}}_{m=m_1^2+\cdots+m_{\ell}^2}1\\
    &=&\sum_{n\geq1}\lambda_f(n)\chi(n)\phi\left(\frac{n}{X}\right)
    \sum_{|m_1|\leq X^{1/2}}\cdots\sum_{|m_{\ell}|\leq X^{1/2}}
    \int_{0}^{1}e((m_1^2+\cdots+m_{\ell}^2-n)\alpha)\mathrm{d}\alpha\\
    &=&\int_0^1 F^{\ell}(\alpha)G(\alpha)\mathrm{d}\alpha.
    \ena
Let $P$ and $Q$ be two parameters satisfying $P<Q\leq X$. Note that $ F^{\ell}(\alpha)G(\alpha)$ is a periodic function of period 1. One further has
    \bna
    \mathcal{S}(X)=\int_{\frac{1}{Q}}^{1+\frac{1}{Q}}F^{\ell}(\alpha)G(\alpha)\mathrm{d}\alpha.
    \ena
By Dirichlet's lemma on rational approximation, each $\alpha\in\left[1/Q,1+1/Q\right]$ can be written in the form
    \bea\label{dirichlet}
    \alpha=\frac{a}{q}+\beta,\qquad |\beta|\le\frac{1}{qQ},
    \eea
where $a$, $q$ are some integers satisfying $1\le a\le q\le Q$, $(a,q)=1$.
We denote by $ \mathfrak{M}(a,q)$ the set of $\alpha$ satisfying \eqref{dirichlet} and define the major arcs $ \mathfrak{M}$ and the minor arcs $\mathfrak{m}$ as follows:
    \[
    \mathfrak{M}=\bigcup_{1\le q\le P}\bigcup_{1\le a\le q\atop (a,q)=1} \mathfrak{M}(a,q),\qquad \mathfrak{m}=\Big[1/Q,1+1/Q\Big] \backslash \mathfrak{M}.
    \]
Then we have
    \bea\label{deposition 1}
    \mathcal{S}(X)=\int_\mathfrak{M}F^{\ell}(\alpha)G(\alpha) \mathrm{d}\alpha+\int_\mathfrak{m}F^{\ell}(\alpha)G(\alpha) \mathrm{d}\alpha.
    \eea

In view of \eqref{deposition 1}, our problem is reduced to evaluating the integrals on the major and  minor arcs.

\subsection{The minor arcs}
The treatment of the minor arcs is primarily based on the following proposition.
    \begin{proposition}\label{minor}
    Let $F(\alpha)$ and $G(\alpha)$ be defined as in \eqref{FG}, and $\alpha\in\mathfrak{m}$. We have
    \bna
    \int_\mathfrak{m}F^{\ell}(\alpha)G(\alpha) \mathrm{d}\alpha\ll X^{\frac{\ell}{2}+\varepsilon}P^{\frac{2-\ell}{2}}
    +X^{\frac{\ell+2}{4}+\varepsilon}+X^{1+\varepsilon}Q^{\frac{\ell-2}{2}},
    \ena
    where the implied constant depends on $f$, $\ell$ and $\varepsilon$.
    \end{proposition}

    \begin{proof}
    By Cauchy-Schwarz inequality, one has
    \bna\label{eq:minor}
    \int_\mathfrak{m}F^{\ell}(\alpha)G(\alpha) \mathrm{d}\alpha
    \ll\sup_{\alpha \in \mathfrak{m}} |F(\alpha)|^{\ell-2}
    \Big(\int_{0}^{1} | F(\alpha) |^4d\alpha\Big)^\frac{1}{2}
    \Big(\int_{0}^{1} |G(\alpha )|^2d\alpha \Big)^\frac{1}{2}.
    \ena
By Weyl's inequality (see \cite[Lemma 2.4]{V}), for $\alpha\in\mathfrak{m}$, we have
    \bna\label{Wely's inequality}
    F(\alpha)
    \ll X^{\frac{1}{2}+\varepsilon}\left(q^{-1}+X^{-\frac{1}{2} }+qX^{-1}\right)^{\frac{1}{2}}
    \ll X^{\frac{1}{2}+\varepsilon}P^{-\frac{1}{2} }+X^{\frac{1}{4}+\varepsilon}+Q^{\frac{1}{2}}.
    \ena
    By Lemma \ref{Hua's lemma}, we have
    \bna\label{Hua}
    \int_{0}^{1}|F(\alpha )|^4 \mathrm{d}\alpha \ll X^{1+\varepsilon}.
    \ena
From Rankin-Selberg estimates \eqref{CSRS}, we deduce that
    \bna
    \int_{0}^{1} |G(\alpha )|^2d\alpha \ll \sum_{n\ll X}|\lambda_f(n)|^{2}\ll X.
    \ena
Combining all estimates above, we complete the proof.\\
\end{proof}

\subsection{The major arcs}
The treatment of the major arcs requires recourse to the Voronoi summation formula, which plays a vital role in the proof of the following proposition.
\begin{proposition}\label{major}
 Let $F(\alpha)$ and $G(\alpha)$ be defined as in \eqref{FG}, and $\alpha\in\mathfrak{M}$. We have
    \bna
    \int_\mathfrak{M}F^{\ell}(\alpha)G(\alpha) \mathrm{d}\alpha
    \ll p\left(1+\frac{X}{PQ}\right)
    \left(\frac{X^{1/2+\ell/2+\varepsilon}P}{Q^{3/2}}
    +\frac{X^{\ell/2+\varepsilon}P^{3/2}\Delta^{1/2}}{Q}\right),
    \ena
 where the implied constant depends on  $f$, $\ell$ and $\varepsilon$.
 \end{proposition}
 \begin{proof}
  On the major arcs, for any $\alpha\in \mathfrak{M}$, we have
    \bea\label{major arcs}
    \int_\mathfrak{M}F^{\ell}(\alpha)G(\alpha)d\alpha
    &=&\sum_{q\le P}\,\,\sideset{}{^*}\sum_{a \bmod q}\int_{\mathfrak{M} (a,q)}F^{\ell}(\alpha)G(\alpha)d\alpha\nonumber\\
    &=&\sum_{q\le P} \,\int_{|\beta |\le \frac{1}{qQ}}\,\sideset{}{^*}\sum_{a \bmod q}F^{\ell}\Big(\frac{a}{q}+\beta\Big)G\Big(\frac{a}{q}+\beta\Big)d\beta,
    \eea
where, throughout the paper the $*$ denotes the condition $(a,q)=1$.
By inserting the definitions of $F(\alpha) $ and $G(\alpha) $ into \eqref{major arcs} and exchanging the order of  the summation over $a$ and the integration over $\beta$, the integration over major arcs is bounded by
    \bea\label{major arcs2}
    &&\sum_{q\leq P}\,\,\int\limits_{|\beta|\leq\frac{1}{qQ}}\sum_{m_1,\cdots,m_{\ell}\ll  X^{1/2}}\Big|\,\,\sideset{}{^*}\sum_{a \bmod q}e\left(\frac{(m_1^2+\cdots+m_{\ell}^2)a}{q}\right)\sum_{n\geq 1}\lambda_f(n)\chi(n)e\left(-\frac{na}{q}\right)\phi_\beta(n)\Big|\mathrm{d}\beta,\nonumber\\
    \eea
where $\phi_\beta(x)=e(-x\beta)w\left(x/X\right)$. By the Fourier expansion of $\chi$ in the terms of additive characters (see \cite[(3.12)]{IK})
    \bna
    \chi(n)=\frac{1}{\tau(\overline{\chi})}\sum_{b \bmod p}\overline{\chi}(b)e\left(\frac{b n}{p}\right),
    \ena
we can transfer the $n$-sum in \eqref{major arcs2} into
    \bna\label{before Voeonoi}
    \sum_{n\geq 1}\lambda_f(n)\chi(n)e\left(-\frac{na}{q}\right)\phi_\beta(n)
    =\frac{1}{\tau(\overline{\chi})}\sum_{b \bmod p}\overline{\chi}(b)\sum_{n\geq 1}\lambda_f(n)e\left(\frac{cn}{pq}\right)\phi_\beta(n),
    \ena
where $c=bq-ap$. Assume $p $ is enough large prime and $ q<P<p$, it is easy to see that $(c, pq) = 1$. By applying Lemma \ref{Voro}, the summation in the absolute value  symbol of \eqref{major arcs2} is equal to
    \bea\label{summation}
    \mathbf{T}:=\frac{pq}{\tau(\overline{\chi})}\sum_{n=1}^{\infty}\frac{\lambda_f(n)}{n}\mathcal{C}(n,m_1^2+\cdots+m_{\ell}^2,q)\Phi_\beta\left(\frac{n}{p^2q^2}\right),
    \eea
where the character sum $\mathcal{C}:=\mathcal{C}(n,m_1^2+\cdots+m_{\ell}^2,q)$ is given by
    \bna
    \mathcal{C}= \sideset{}{^*}\sum_{a \bmod q}
    e\left(\frac{(m_1^2+\cdots+m_{\ell}^2)a}{q}\right)
    \sum_{b \bmod p}\overline{\chi}(b)e\left(-\frac{\overline{c}n }{pq}\right)
    \ena
and
    \bna
    \Phi_\beta(x)= i^{\kappa-1}\frac{1}{2\pi^2}\int_{(\sigma)}(\pi^2 x)^{-s}\rho_f (s)\widetilde{\phi}_\beta(-s)\mathrm{d}s.
    \\\ena

We first estimate the character sum $\mathcal{C}$. Since $(p,q)=1$, we compute
    \bna
    \mathcal{C}
    &=&\sideset{}{^*}\sum_{a \bmod q}e\left(\frac{(m_1^2+\cdots+m_{\ell}^2)a}{q}\right)\sum_{b \bmod p}\overline{\chi}(b)e\left(-\frac{n \overline{ bq^2}}{p}+\frac{n \overline{ ap^2}}{q}\right)\\
    &=&S(m_1^2+\cdots+m_{\ell}^2, n\overline{p^2};q)\sum_{b \bmod p}\overline{\chi}(b)e\left(-\frac{n \overline{ bq^2}}{p}\right)\\
    &=&S(m_1^2+\cdots+m_{\ell}^2, n\overline{p^2};q)\overline{\chi}(-n\overline{ q^2})\tau(\chi).
    \ena
By the estimate  $|\tau(\chi)|=\sqrt{p}$ and Weil's bound for Kloosterman sums, one has
    \bea\label{eq-charractersum}
    \mathcal{C}
    &\ll& (m_1^2+\cdots+m_{\ell}^2,n,q)^{\frac{1}{2}}
    q^{\frac{1}{2}}\tau(q)p^{\frac{1}{2}}\nonumber\\
    &\ll& p^{\frac{1}{2}}q^{\frac{1}{2}+\varepsilon}(n,q)^{\frac{1}{2}}.
    \eea

Next we estimate the integral $\Phi_\beta(x)$. Recall that $\phi_\beta(x)=e(-x\beta)w\left(\frac{x}{X}\right)$,
where $w(x)$ is a smooth function compactly supported in $[1/2, 1]$ and satisfying
    \bna
    w^{(j)}(x)\ll_j \Delta^j \quad\text{ for}\quad j\geq 0 \quad
    \text{and} \quad
    \int|w^{(j)}(\xi)|\mathrm{d}\xi\ll \Delta^{j-1} \quad \text{for}\quad j\geq1.
    \ena
Thus we have
    \bna
    \phi_\beta^{(j)}(x)\ll\left(\frac{\Delta+|\beta|X}{X}\right)^j
    \ena
for any $j\geq0$  and
    \bna
    \int_{0}^{\infty}|\phi_\beta^{(j)}(x)|\mathrm{d}x
    &=&\int_{0}^{\infty}\left|\sum_{0\leq v\leq j}C_{j}^{v}(-2\pi i\beta)^{j-v}e(-\beta x)w^{(v)}\left(\frac{x}{X}\right)\frac{1}{X^v}\right|\mathrm{d}x\\
    &\ll&|\beta|^jX+\sum_{1\leq v\leq j}C_{j}^{v}|\beta|^{j-v}X^{-v+1}
    \int_{0}^{\infty}|w^{(v)}(\xi)|\mathrm{d}\xi\\
    &\ll&|\beta|^{j}X+|\beta|^{j-1}\sum_{1\leq v\leq j}\left(\frac{\Delta}{|\beta|X}\right)^{v-1}\\
    &\ll&|\beta|^{j-1}+|\beta|^{j}X+\left(\frac{\Delta}{X}\right)^{j-1}\\
    &\ll&(1+|\beta|X)\left(\frac{|\beta|X+\Delta}{X}\right)^{j-1}
    \ena
for any $j\geq 1$. It follows from Lemma \ref{evaluation of G} that
    \bea\label{eq:integral}
    \Phi_{\beta}(x)=
    \left\{
    \begin{array}{ll}
    X^{-A} , & \text{if }\, x>(\Delta+|\beta|X)^{2+\varepsilon}X^{-1} ,\\
    (1+|\beta|X)(xX)^{\frac{1}{4}},
    & \text{if }\, X^{-1}\ll x \leq (|\beta|X+\Delta)^{2+\varepsilon}X^{-1} ,\\
    (1+|\beta|X)(xX)^{\frac{1}{2}}( \Delta+|\beta|X)^{\varepsilon},
    & \text{if }\, x \ll X^{-1}.
    \end{array} \right.
    \eea

Plugging the estimate \eqref{eq-charractersum} into \eqref{summation}, we obtain
    \bea\label{tt}
    \mathbf{T}
    \ll\frac{p^{\frac{3}{2}}q^{\frac{3}{2}+\varepsilon}}{\tau(\overline{\chi})}
    \sum_{n=1}^{\infty}\frac{\lambda_f(n)(n,q)^{\frac{1}{2}}}{n}\Phi_\beta\left(\frac{n}{p^2q^2}\right).
    \eea
In view of the Ramanujan conjecture \eqref{eq:KS} and the Rankin-Selberg estimates \eqref{CSRS}, one has
    \bna
    \sum_{n\leq x}|\lambda_f(n)|(n,q)^{\frac{1}{2}}
    &\leq& \sum_{l|q}l^{\frac{1}{2}}\sum_{n\leq x/l}|\lambda_f(nl)|\\
    &\leq& \sum_{l|q}l^{\frac{1}{2}}\sum_{r|l^\infty}\sum_{n\leq x/(rl)\atop (n,r)=1}|\lambda_f(rnl)|\\
    &\leq& \sum_{l|q}l^{\frac{1}{2}}\sum_{r|l^\infty}| \lambda_f(rl)|\sum_{n\leq x/(rl)}|\lambda_f(n)|\\
    &\leq& x^{1+\varepsilon}.
    \ena
By virtue of \eqref{eq:integral},  the summation of \eqref{tt} is bounded by
    \bna\label{after integral}
    &\ll&(1+|\beta|X)X^\varepsilon\sum_{n=1}^{\infty}\frac{|\lambda_f(n)|(n,q)^{\frac{1}{2}}}{n}
    \max_{X^{-1}\ll \frac{n}{p^2q^2} \leq (|\beta|X+\Delta)^{2+\varepsilon}X^{-1} } \left(\frac{nX}{p^2q^2}\right)^{\frac{1}{4}}\nonumber\\
    &&+(1+|\beta|X)X^\varepsilon
    \sum_{n=1}^{\infty}\frac{|\lambda_f(n)|(n,q)^{\frac{1}{2}}}{n}
    \max_{\frac{n}{p^2q^2} \ll X^{-1} }
    \left(\frac{nX}{p^2q^2}\right)^{\frac{1}{2}}\nonumber\\
    &\ll& (1+|\beta|X)X^\varepsilon
    \max_{ \frac{p^2q^2}{X}\ll N\ll\frac{(pq(|\beta|X+\Delta))^{2+\varepsilon}}{X}}
    N^{-1}\left(\frac{NX}{p^2q^2}\right)^{\frac{1}{4}}
    \sum_{n\sim N}|\lambda_f(n)|(n,q)^{\frac{1}{2}}\nonumber\\
    &&+(1+|\beta|X)X^\varepsilon
    \max_{N\ll\frac{p^2q^2}{X}}
    N^{-1}\left(\frac{NX}{p^2q^2}\right)^{\frac{1}{2}}
    \sum_{n\sim N}|\lambda_f(n)|(n,q)^{\frac{1}{2}} \nonumber\\
    &\ll& (1+|\beta|X)(|\beta|X+\Delta)^{\frac{1}{2}}X^\varepsilon.
    \ena
Hence, we obtain that \eqref{tt} is controlled by
    \bna
    \mathbf{T}\ll pq^{\frac{3}{2}+\varepsilon}
    (1+|\beta|X)(|\beta|X+\Delta)^{\frac{1}{2}}X^\varepsilon.
    \ena
Inserting this into \eqref{major arcs2}, we derive the integration over major arcs
    \bna\label{major arcs result}
    \int_\mathfrak{M}F^{\ell}(\alpha)G(\alpha)d\alpha
    \ll p\left(1+\frac{X}{PQ}\right)
    \left(\frac{X^{1/2+\ell/2+\varepsilon}P}{Q^{3/2}}
    +\frac{X^{\ell/2+\varepsilon}P^{3/2}\Delta^{1/2}}{Q}\right).
    \ena
This completes the proof of Proposition \ref{major}.\\
\end{proof}

\subsection{Conclusion}
Assume that the parameters P and Q satisfy
    \bna
    PQ=X.
    \ena
By Proposition \ref{minor}, we have
    \bna
    \int_\mathfrak{m}F^{\ell}(\alpha)G(\alpha) \mathrm{d}\alpha\ll X^{1+\varepsilon}Q^{\frac{\ell-2}{2}}
    +X^{\frac{\ell+2}{4}+\varepsilon}.
    \ena
By Proposition \ref{major}, we have
    \bna
    \int_\mathfrak{M}F^{\ell}(\alpha)G(\alpha) \mathrm{d}\alpha
    \ll pX^{\frac{3}{2}+\frac{\ell}{2}+\varepsilon}Q^{-\frac{5}{2}}\Delta^{\frac{1}{2}}.
    \ena
Hence, we obtain
    \bna
    \mathcal{S}(X)&=&
    \int_\mathfrak{M}F^{\ell}(\alpha)G(\alpha) \mathrm{d}\alpha
    +\int_\mathfrak{m}F^{\ell}(\alpha)G(\alpha) \mathrm{d}\alpha\\
    &\ll&pX^{\frac{3}{2}+\frac{\ell}{2}+\varepsilon}Q^{-\frac{5}{2}}\Delta^{\frac{1}{2}}
    +X^{1+\varepsilon}Q^{\frac{\ell-2}{2}}
    +X^{\frac{\ell+2}{4}+\varepsilon}.
    \ena

With a suitable choice of
    \bna
    Q=p^{\frac{2}{\ell+3}}X^{\frac{\ell+1}{\ell+3}}\Delta^{\frac{1}{\ell+3}},
    \ena
we  obtain
    \bea\label{re}
    \mathcal{S}(X)\ll
    p^{\frac{\ell-2}{\ell+3}}X^{\frac{\ell}{2}-\frac{\ell-2}{\ell+3}+\varepsilon}
    \Delta^{\frac{\ell-2}{2(\ell+3)}}.
    \eea
We finish the proof of Theorem \ref{Th11}.

\section{Proof of Theorem \ref{Th12}}
Our task now is to show Theorem \ref{Th12}. Recall that $w(x)$ is a smooth function compactly supported in $[1/2, 1]$, say, such that $w(x)=1$ for $1/2+1/(8\Delta)\leq x\leq1-1/(8\Delta)$ and $
 w^{(j)}(x)\ll_j \Delta^j $ for $ j\geq 0$. Inserting the
individual bound \eqref{r} for $r_\ell(n)$, we can  write
    \bea\label{eq:th2}
    \sum_{X/2<n\leq X}\lambda_f(n)r_{\ell}(n)\chi(n)
    =\mathcal{S}(X)+O\left(X^{\frac{\ell}{2}-1}\Big(\sum_{X/2<n\leq X/2+X/(8\Delta)}+\sum_{X-X/(8\Delta)<n\leq X}\Big)|\lambda_f(n)|\right).\nonumber\\
    \eea
In view of \eqref{eq:th2}, we need to consider the short interval sum
    \bna
    \sum_{x<n\leq x+y}|\lambda_f(n)|,
    \ena
where $y\ll x$. The Rankin-Selberg estimates \eqref{CSRS} implies
    \bna
    \sum_{x<n\leq x+y}|\lambda_f(n)|^2\ll x.
    \ena
By means of Cauchy-Schwarz inequality, we deduce that
    \bna
    \sum_{x<n\leq x+y}|\lambda_f(n)|\ll (xy)^{\frac{1}{2}+\varepsilon}.
    \ena
Therefore, the $O$-term in \eqref{eq:th2} is controlled by
    \bna
    \ll X^{\frac{\ell}{2}}\Delta^{-\frac{1}{2}}.
    \ena
Combining \eqref{re} with the estimates of $\lambda_f(n)$ in short intervals, we obtain
    \bna
    \sum_{X/2<n\leq X}\lambda_f(n)r_{\ell}(n)\chi(n)
    \ll p^{\frac{\ell-2}{\ell+3}}\Delta^{\frac{\ell-2}{2(\ell+3)}}
    X^{\frac{\ell}{2}-\frac{\ell-2}{\ell+3}+\varepsilon}
    +X^{\frac{\ell}{2}}\Delta^{-\frac{1}{2}}.
    \ena
By choosing
    \bna
    \Delta=p^{\frac{2(2-\ell)}{2\ell+1}}X^{\frac{2(\ell-2)}{2\ell+1}},
    \ena
we derive
    \bea\label{re2}
    \sum_{X/2<n\leq X}\lambda_f(n)r_{\ell}(n)\chi(n)\ll p^{\frac{\ell-2}{2\ell+1}}X^{\frac{\ell}{2}-\frac{\ell-2}{2\ell+1}}.
    \eea
Theorem \ref{Th12} follows by splitting the sum over $n$ into sums over dyadic intervals and applying \eqref{re2} to each of them.

\end{document}